\documentclass[dvipdfmx,autodetect-engine]{article}
\usepackage{design_ASC}

\title{Versality of Rotation Unfolding of Folding Maps for Surfaces in $\mathbb{R}^3$} 

\author{Toshizumi Fukui and Atsuki Hiramatsu\\ 
    \footnotesize
    Department of Mathematics, Saitama University\\
    \footnotesize
    255 Shimo-Okubo, Sakura-ku, Saitama 338-8570, Japan\\
    \footnotesize
    E-mail addresses: \texttt{tfukui@rimath.saitama-u.ac.jp}, \texttt{a.hiramatsu.2357@gmail.com}
}

\date{\today} 

\begin{document}
\setlength{\droptitle}{-5em}    
\maketitle

\begin{abstract}
    We introduce the rotation unfolding of the folding map of a surface in $\mathbb{R}^3$, and investigate its $\mathcal{A}$-vesality. The rotation unfolding is a 2-parameter unfolding and can be considered as a subfamily of the folding family, which is introduced by Bruce and Wilkinson. They revealed relationships between a bifurcation set of this family and the focal/symmetry set of a surface in $\mathbb{R}^3$. We state the criteria of singularities of the folding map up to codimension 2 and prove when our rotation unfolding is versal. The conditions to be versal are stated in terms of geometry. As a by-product, we show the diffeomorphic type of the locus of the tangent planes of the focal set of regular surfaces, which passes through the origin. 
\end{abstract}



\section{Introduction}\label{S:Int}
    
    Let $M$ be a regular surface in $\mathbb{R}^3$, expressed by $z = f(x,y)$. Composing the map $(x,y,z) \mapsto (x,y^2,z)$ with the graph map $(x,y) \mapsto (x,y,f(x,y))$, we obtain the folding map $F:\mathbb{R}^2 \rightarrow \mathbb{R}^3$ of the surface $M$, which is given by 
    $$F(x,y) = (x,y^2,f(x,y)).$$
    
    
    \begin{center}
        \includegraphics{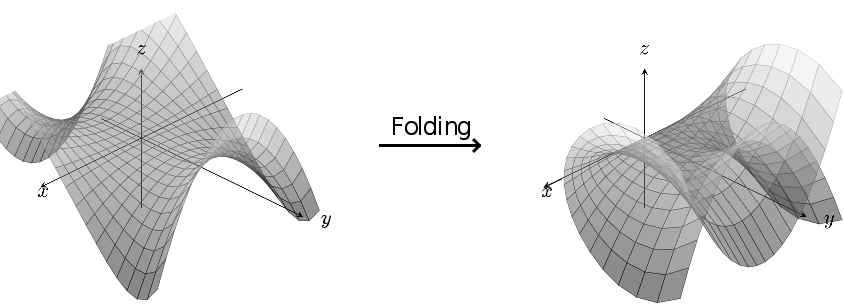}
\end{center}

    It is known that the folding map $F$ has a singularity at $\bm{0} \in \mathbb{R}^2$ whether the folding direction (in this case, $y$-direction) at the origin is in the tangent plane at $\bm{0}$. The singularity of $F$ at $\bm{0}$ is a cross-cap if the $y$-direction is not a principal direction of $M$ at $\bm{0}$. In other word, the singularity of $F$ at $\bm{0}$ is more degenerated than a cross-cap if the $y$-direction is a principal direction of $M$ at $\bm{0}$.

    Bruce and Wilkinson (\cite{BW}) introduced the notion of the folding family, which is obtained by conjugating the fold map $(x,y,z) \mapsto (x,y^2,z)$ by Euclidean motions and restricting this family to the surface $M$. The family is a natural unfolding of the folding map $F$ with respect to Euclidean motions of $\mathbb{R}^3$. Bruce and Wilinson show the classification of singularities of the folding map $F$, and that the folding family is a versal unfolding of $F$, with respect to $\mathcal{A}$-equivalence, for a residual set of embeddings of $M$ in $\mathbb{R}^3$.

    The notion of the folding family is motivated by describing infinitesimal reflectional symmetries of surfaces. Usually, the reflection symmetry of the surface $M$ in the plane $y=0$ means that $f(x,y)=f(x,-y)$ holds. So infinitesimal reflectional symmetries would be something which implies $f_o(x,y) = \frac{1}{2}(f(x,y)-f(x,-y))$ is close to 0 near the origin. Remark that $f_o(x,y)$ is the odd part of $f(x,y)$ with respect to $y$. A formulation for infinitesimal reflectional symmetries is as follows:
    \begin{itemize}
        \item When $f_o(x_0,y_0) = \pardiff{}{x}f_o(x_0,y_0) = \pardiff{}{y}f_o(x_0,y_0) = 0$ holds at a point $(x_0,y_0)$, then $F(x_0,y_0)= F(x_0,-y_0)$ and the image of tangent map of $F$ at $(x_0,y_0)$ and that of $F$ at $(x_0,-y_0)$ coincide. We refer to such a point $F(x_0,y_0)$ as a self-tangency point of $F$. 

        
        \item If $F$ is $\mathcal{A}$-equivalent to $B_2$ singularity, $(x,y) \mapsto (x,y^2,y^5-x^2y)$, at the origin, there is a self-tangency point for several small perturbations of $F$. 
    \end{itemize}

    \begin{figure}[h!]
        \centering
        \includegraphics[scale=0.4]{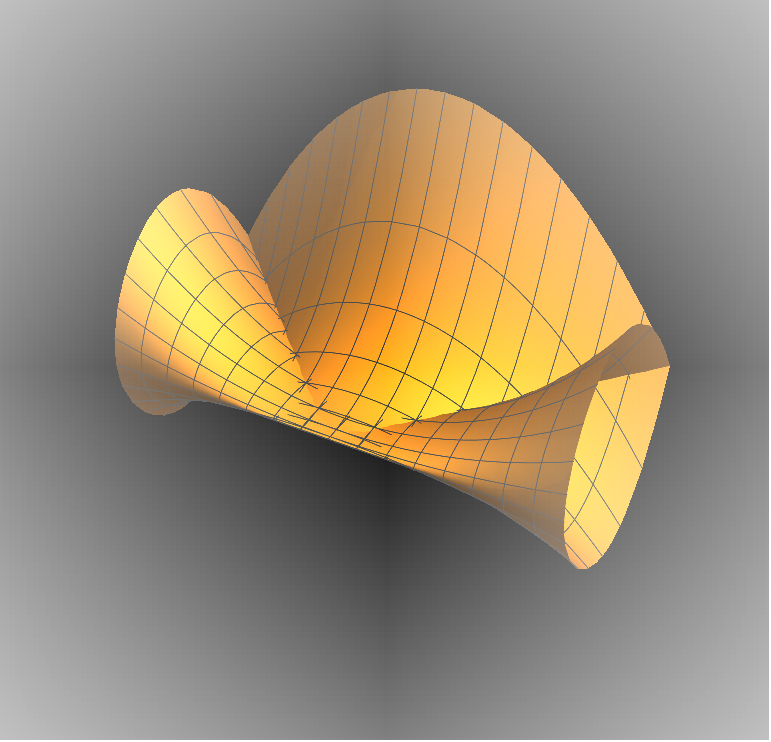}
        \includegraphics[scale=0.4]{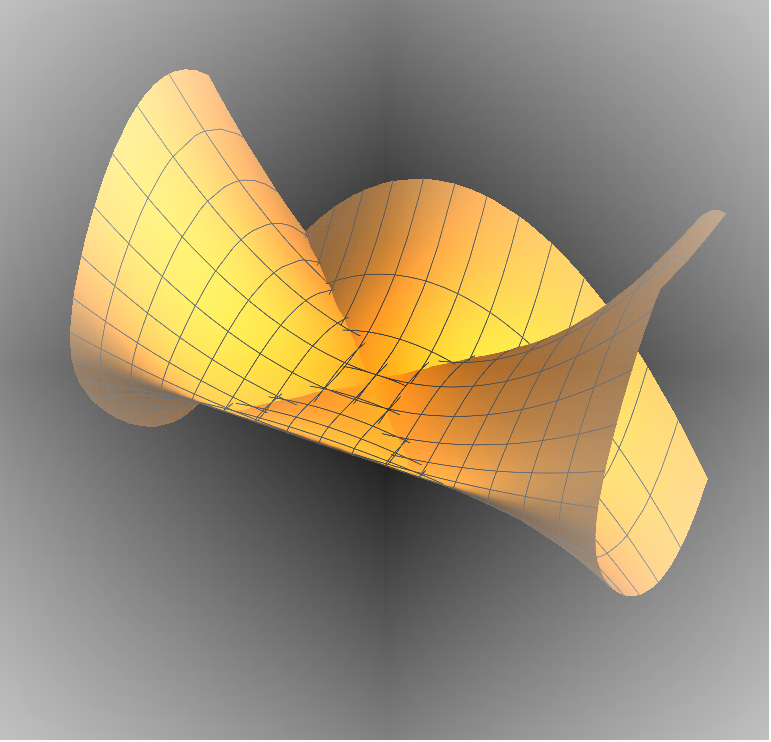}\\
        \vspace{-3.5em}$a=0$ \hspace{12em} $a=0.6$
        \caption{$(x,y^2,y^5-x^2y+a^4y-2a^2y^3)$}
        \label{fig:enter-label}
    \end{figure}
    
    This observation comes from the viewpoint that to investigate singularities of the folding map $F$ is to describe infinitesimally reflectional symmetry of $M$, due to Bruce and Wilkinson (\cite{BW}). They also show that the bifurcation set of the folding family is the dual of the focal and symmetry sets of the surface $M$.
    
    In this paper, we study a natural subfamily of the folding family obtained by restricting Euclidean motions to rotations. We call this family the rotation unfolding of $F$ denoted by $R:\mathbb{R}^2 \times S^2 \rightarrow \mathbb{R}^3$. (See section \ref{S:Def}.)  According to the list of singularities of the folding map $F$ due to Bruce and Wilkinson (\cite{BW}), we expect a generic two-parameter family to have the following singularities.


     
    \begin{center}
        \begin{tabular}{ccl}\hline
            Name & Standard form    & Geometric condition \\ \hline
            $S_1$  & $(x,y^2,y^3 \pm x^2y)$  & general smooth point of a focal set  \\
            $S_2$  & $(x,y^2,y^3 \pm x^3y)$   & parabolic smooth point of a focal set   \\
            $B_2$  & $(x,y^2,y^5 \pm x^2y)$   & general cusp point of a focal set\\ \hline
        \end{tabular}
    \end{center}

    For folding map $F$, we show the criteria for these singularities in Proposition.\ref{prop:AC}..



    We investigate criteria for the versality of the rotation unfolding. Our main theorem is as follows.

    \begin{thm*}
        Let the Taylor expansion of $f(x,y)$ be $\frac{1}{2} (k_1 x^2 + k_2 y^2) + \sum^m_{i+j \geq 3} \frac{1}{i!j!}a_{ij} x^i y^j + O(m+1)$ where $m$ is an integer $\geq 3$.\\
        (1) If $F$ is $S_0$ or $S_1$ singularity at $\bm{0}$, $R$ is always versal.\\
        (2) If $F$ is $S_2$ singularity at $\bm{0}$, $R$ is versal if and only if the origin is not umbilic.\\ 
        (3) If $F$ is $B_2$ singularity at $\bm{0}$, $R$ is versal if and only if the origin is not umbilic and ridge line transverse to the plane $\Pi$, or $D_4$ type umbilic.
    \end{thm*}


    As a by-product, we conclude the diffeomorphic type of the locus of the tangent planes of the focal set of the surface $M$, which passes through the origin. (\S\ref{S:BofR})
    

    We will discuss versality of the folding family in the forthcoming paper (\cite{FH}).

    The paper is organized as follows. In $\S$\ref{S:Pre}, we recall some notions and theorems to prove the main theorem. In $\S$\ref{S:Def}, we define the rotation unfolding and give it in an explicit form. In $\S$\ref{S:Ver}, we show the main theorem for algebraic condition. In $\S$ \ref{S:Geo}, we consider geometric meaning for this conditions.
    
    The authors are grateful to Farid Tari since a discussion with him motivates the paper.


\section{Preliminaries}\label{S:Pre}

    \subsection{Geometry of surfaces}

        First, we recall the notion of a ridge line and a subparabolic line (see \cite{Geo}). 

        \begin{dfn}     
            Let $p$ be non umbilical point of a regular surface with principal vectors $\bm{v}_1,\bm{v}_2$, and the corresponding principal curvatures $\kappa_1, \kappa_2$, which are defined near $p$. \\
            (1) We say that the point $p$ is a $\bm{v}_i$-ridge point if $\bm{v}_i \kappa _i(p) = 0$, where $\bm{v}_i \kappa _i$ is the directional derivative of $\kappa_i$ in $\bm{v}_i$. The closure of the set of $\bm{v}_i$-ridge points is called a $\bm{v}_i$-ridge line if it is of one-dimensional.\\
            (2) We say that the point $p$ is a $\bm{v}_i$-subparabolic point if $\bm{v}_i \kappa _j(p) = 0 \ (i \not = j)$. The closure of the set of $\bm{v}_i$-subparabolic points is called a $\bm{v}_i$-subparabolic line if it is of one-dimensional.\\           
        \end{dfn}

        \begin{lem}[\cite{F}, Lemma 2.1]     
            Let $g(u,v) = (u,v,f(u,v))$ denote a surface where 
                $$f(u,v) = \frac{1}{2} (k_1 u^2 + k_2 v^2) + \sum^m_{i+j \geq 3} \frac{1}{i!j!}a_{ij} u^i v^j + O(m+1),$$
            $m$ is an integer $\geq 3$, and $g(0)$ be not umbilic, i.e., $k_1 \not = k_2$.\\
            (1) The origin is $\bm{v}_2$-ridge if and only if $a_{03} = 0$.\\
            (2) The origin is $\bm{v}_2$-subparabolic if and only if $a_{21} = 0$.
        \end{lem}

        The asymptotic expansions of $\bm{v}_2 \kappa_2(u,v), \bm{v}_2 \kappa_1(u,v)$, appeared in the proof of the lemma above, are important, and we state them next as a lemma.

        \begin{lem} \label{prop:ridge} 
            $$
            \begin{array}{rccl}
                (1) & \bm{v}_2 \kappa_2(u,v) &= & a_{03} + \frac{1}{k_2-k_1} \{ 3a_{21}a_{12} + a_{13}(k_2-k_1)\}u \\
                 & & & \ \ \ \ \ \ \ \ \ \ \ \ \ \ \ + \frac{1}{k_2-k_1} \{ 3a_{12}^2 + (a_{04}-3k_2^3)(k_2-k_1)\}v + O(2)\\
                (2) & \bm{v}_2 \kappa_1(u,v) &= & a_{21} + \frac{1}{k_1-k_2} \{ a_{21}(a_{12}-a_{30}) + a_{31}(k_1-k_2)\}u \\
                 & & & \ \ \ \ \ \ \ \ \ \ \ \ \ \ \ + \frac{1}{k_1-k_2} \{ a_{12}(2a_{12}-a_{30}) + (a_{22}-k_1k_2^2)(k_1-k_2)\}v + O(2)
            \end{array}
            $$
        \end{lem}

        If the origin is an umbilic point, then we are not able to define ridge and subparabolic as above.
        That is why we have to consider the geometrical conditions separately when $g(o)$ is umbilic, i.e., $k_1 = k_2 =: k$. We express $g$ using complex coordinates $z=x+iy$, and obtain the following:
        
        $$f(z,\bar{z}) = \frac{k}{2} z\bar{z} + c(z) + O(4), \text{ where } c(z)=\frac{1}{6}\text{Re}(\alpha z^3 + 3\beta z^2\overline{z}).$$

        For generic umbilics, the following properties are known.

        \begin{prop}[\cite{BW}, \cite{CP}]\label{prop:umb}   
            We assume that $g(0)$ is umbilic of a regular surface $g$.
            
                \centering
                    \begin{tabular}{cll} \hline
                        &  & Conditions\\ \hline
                        \ronumi & The umbilic is symmetric or unsymmetric in the inner of the \ronumii.   & $3|\beta| \not=|\alpha|$\\
                        & \ \ \ (The $c(z)$ dose not have orthogonal roots.)\\
                        
                        \ronumii & The umbilic is of type $D_4$.      & $3\beta \not= -(2\alpha e^{2i\theta} + \overline{\alpha} e^{-4i\theta})$  \\
                        & \ \ \ (The discriminant of $c(z)$ is not 0.)  & ($0 \leq \theta \leq 2\pi$)\\
                        
                        \ronumiii & The umbilic is star, monster, or lemon. & $|\beta | \not= |\alpha|$\\
                        & \ \ \ ($x^2+y^2$, $c_x$, $c_y$ are linearly independent.) \\
                        
                        \ronumiv & There are one or three limiting principal directions. & $\beta \not= -(2 e^{2i\theta} + e^{-4i\theta})$\\
                        & \ \ \ (The discriminant of $i(zc_z - \overline{z}c_{\overline{z}})$ is not 0.)   & ($0 \leq \theta \leq 2\pi$)\\
                        
                        \ronumv & The origin is not in a subparabolic line and a ridge line. & $\text{arg}(\alpha) \not= \text{arg}(\beta)$\\ 
                        & \ \ \ ($\ c_z=c_{\bar{z}} = 0$ does not hold)      \\\hline
                    \end{tabular}

            
            \begin{figure}[h!]
                \centering
                \includegraphics{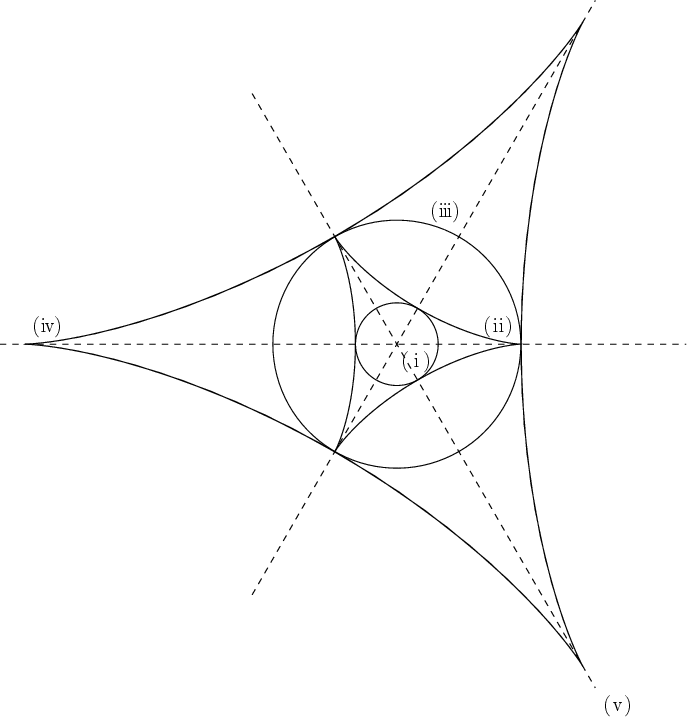}
                \caption{Illustration of the umbilic conditions in $\beta$-plane $\alpha=1$}
            \end{figure}
        \end{prop}






    \subsection{The finite determinacy and versality}
        In this section, we recall Mather's theory on versal unfolding (see \cite{Mtext}).
        Let $f:\mathbb{R}^n,\bm{0} \rightarrow \mathbb{R}^p,\bm{0}$ be a $C^\infty$-germ, $\theta_k$ be the set of germs of $C^\infty$-sections $\mathbb{R}^k, \bm{0}  \rightarrow T\mathbb{R}^k$, and $\theta(f)$ be the set of vector fields along $f$. We consider two important maps $tf:\theta_n \rightarrow \theta(f)$ and $\omega f:\theta_p \rightarrow \theta(f)$, which are defined by $\xi \mapsto df\circ\xi$ and $ \eta \mapsto \eta\circ f$, respectively. And let $T\mathcal{A}_e f = tf(\theta_n) + \omega f(\theta_p)$.


        \begin{dfn}
            Let $f:\mathbb{R}^n,\bm{0}  \rightarrow \mathbb{R}^p,\bm{0}$ be a $C^\infty$-germ.\\
            (1) The {\bf codimension} $\text{codim}_{\mathcal{A}_e} f$ is defined as follows.
            $$\text{codim}_{\mathcal{A}_e} f := \dim_{\mathbb{R}}\fracset{\theta(f)}{tf(\theta_n) + \omega f(\theta_p)} = \dim_{\mathbb{R}}\fracset{\theta(f)}{T\mathcal{A}_ef}$$
            (2) We say that $f$ is {\bf $k$-determined} if any germ $g:\mathbb{R}^n,\bm{0} \rightarrow \mathbb{R}^p,\bm{0}$, having the same $k$-jets to $f$ at $\bm{0}$, is $\mathcal{A}$-equivalent to $f$. We say that $f$ is finitely $\mathcal{A}$-determined if it is $k$-determined for some $k>0$. The $\mathcal{A}$-determinacy degree of $f$ is the lowest $k$ such that $f$ is  $k$-determined.
        \end{dfn}

    \begin{thm}[\cite{Mtext}, Theorem 6.2]\label{prop:det}
        If $f$ is $k$-determined, then $m_n^{k+1} \theta(f) \subset T\mathcal{A}_ef$, where $m_n = \{ \phi:\mathbb{R}^n, \bm{0} \rightarrow \mathbb{R},C^{\infty}, \phi(\bm{0})=0 \}$.
    \end{thm}

    \begin{dfn}
        Let $F:\mathbb{R}^n\times\mathbb{R}^r,\bm{0} \ni (x,u) \mapsto (\tilde{f}(x.u),u) \in \mathbb{R}^p\times\mathbb{R}^r,\bm{0}$ be an unfolding of $f_0:\mathbb{R}^n,\bm{0} \rightarrow \mathbb{R}^p,\bm{0}$. We say that $F$ is {\bf versal} if for every unfolding $G$ of $f_0$ with parameter space $(\mathbb{R}^s,\bm{0})$, there is a $C^\infty$-map $h:\mathbb{R}^s,\bm{0} \rightarrow \mathbb{R}^r,\bm{0}$ such that G is equivalent as unfolding to the unfolding $h^{*}F$, where $h^*F(x,v) = (\tilde{f}(x,h(v)),v)$ known as pull-back by base-change map.
    \end{dfn}

    \begin{thm}[\cite{Mtext}, Theorem 5.1]\label{prop:IV}
        The unfolding $F$ of $f_0$ is versal if and only if
        $$T\mathcal{A}_ef_0 + V_F = \theta(f_0).$$
        Here, $V_F:= \langle \frac{\partial F}{\partial u_1}|_{\mathbb{R}^n \times \bm{0}},...,\frac{\partial F}{\partial u_r}|_{\mathbb{R}^n \times \bm{0}} \rangle _{\mathbb{R}}$ is an $\mathbb{R}$-module spanned by partial differentiations of $F$ with respect to unfolding parameters with restriction to $\mathbb{R}^n\times\{\bm{0}\}$.
    \end{thm}

    \subsection{The classification of germs of maps from $\mathbb{R}^2$ to $\mathbb{R}^3$}

        We use the classification result of map germs from $\mathbb{R}^2$ to $\mathbb{R}^3$ due to Mond (\cite{Mclass} and \cite{F}).
        

        \begin{thm}[\cite{Mclass}, Theorem 1.1]\label{prop:MC}    
            The $\mathcal{A}$-simple map $(x,y^2,f(x,y))$ is $\mathcal{A}$-equivalent to one of the germs on the following list.

            \begin{center}
                \begin{tabular}{cccc}\hline
                    $\text{Germ} \  f(x,y)$   & $\mathcal{A}_e \text{-codim}$   & $\text{Determinacy for }\mathcal{A}$ & $\text{Name}$\\ \hline
                    $(x,y^2,xy)$  & $0$ & $2$     & Cross-cap ($S_0$)\\
                    $(x,y^2,y^3 \pm x^{k+1}y),\ k \geq 1$     & $k$   & $k+2$   & $S_k^{\pm}$\\
                    $(x,y^2,x^2y \pm y^{2k+1}), \ k \geq 2$   & $k$ & $2k+1$     & $B_k^{\pm}$\\
                    $(x,y^2,xy^3 \pm x^ky), \ k \geq 3$   & $k$ & $k+1$     & $C_k^{\pm}$\\
                    $(x,y^2,x^3y + y^5)$  & $4$ & $5$     & $F_4$\\     \hline
                \end{tabular}
            \end{center}
        \end{thm}

        For an unfolding to be versal, it is necessary for it to have parameters at least as many as the codimension of each singularity. The singularities with codimension 2 or less are $S_0, S_1, S_2$ and $B_2$. Therefore, the rotation unfolding can be versal only if the folding map $F$ has a singularity $\mathcal{A}$-equivalent to them.

        \begin{prop}\label{prop:AC}    
            Let $f(x,y) = \sum^m_{i+j \geq 2} \frac{1}{i!j!}a_{ij} x^i y^j + O(m+1)$. We show the algebraic conditions for $f(x,y)$ such that the folding map $(x,y^2,f(x,y))$ is $\mathcal{A}$-equivalent to the following singularities as follows.

            \begin{center}
                \begin{tabular}{cl}\hline
                    $\text{Name}$ & $\text{Algebraic condition}$ \\ \hline
                    Cross-cap($S_0$) & $a_{11} \not = 0$ \\
                    $S_1$  & $a_{11} = 0, a_{21} \not = 0, a_{03} \not = 0$  \\
                    $S_2$   & $a_{11} = 0, a_{21} = 0, a_{03} \not = 0, a_{31} \not = 0$ \\
                    $B_2$   & $a_{11} = 0, a_{21} \not = 0, a_{03} = 0, 3a_{21}a_{05}-5a_{13}^2 \not = 0$  \\ \hline
                \end{tabular}
            \end{center}
        \end{prop}

        \begin{proof}
            See the proof of Proposition 2.2 in \cite{FH}. Refer Table 8 in \cite{F} also. 
        \end{proof}


        The geometric conditions of these singularities at a non umbilical point are as follows:
        
        \begin{itemize}
            \item  If the folding map $F$ is $\mathcal{A}$-equivalent to the $S_1$ singularity, a principal vector $v$ of the surface $M$ at the origin is orthogonal to the reflection plane $\Pi$ and the origin is not $v$-ridge and $v$-subparabolic. 
            (If the origin is umbilic, no $v$-ridge lines reach the origin and no $v$-subparabolic lines reach the origin.)
        
            \item If the folding map $F$ is $\mathcal{A}$-equivalent to the $S_2$ singularity, a principal vector $v$ of the surface $M$ at the origin is orthogonal to the reflection plane $\Pi$ and the origin is $v$-subparabolic and not $v$-ridge. 
            (If the origin is umbilic, no $v$-ridge lines reach the origin and a $v$-subparabolic line reaches the origin.)
        
            \item If the folding map $F$ is $\mathcal{A}$-equivalent to the $B_2$ singularity, a principal vector $v$ of the surface $M$ at the origin is orthogonal to the reflection plane $\Pi$ and the origin is $v$-ridge and not $v$-subparabolic.
            (If the origin is umbilic, a $v$-ridge line reaches the origin and no $v$-subparabolic lines reach the origin.)
            
        \end{itemize}


\section{Rotation Unfolding of Folding maps}\label{S:Def}

    The rotation unfolding is obviously versal if the singularity of $F$ is a cross-cap. Thus, we can assume that the coefficient $a_{11}$ is $0$, that is, we can assume that $f$ is given in the following form:
    
    $$f(x,y) = \frac{1}{2} (k_1 x^2 + k_2 y^2) + \sum^m_{i+j \geq 3} \frac{1}{i!j!}a_{ij} x^i y^j + O(m+1).$$
        
    Let $\Pi_{\bm{v}}$ be a plane through $(0,0,0)$ with a normal vector $\bm{v} \in S^2$. Let $\bm{\nu}=(0,0,1)$ be a normal vector of the surface $M$ at $\bm{0}$. We consider an orthonormal frame
    
    $$\bm{\nu} \times \bm{v}, \ \bm{v} ,\  (\bm{v} \times \bm{\nu}) \times \bm{v}.$$
    Then the folding map for $\bm{v}$-direction is given by
        \begin{equation}{\label{FoldingMap}}
            s \bm{v} \times \bm{\nu} + t \bm{v} + r (\bm{v} \times \bm{\nu}) \times \bm{v} \ \longmapsto\  s \bm{v} \times \bm{\nu} + t^2 \bm{v} + r (\bm{v} \times \bm{\nu}) \times \bm{v}.     
        \end{equation}
    The rotation unfolding is obtained as a composition of \eqref{FoldingMap} with graph map $g:(x,y) \mapsto (x,y,f(x,y))$. The following lemma allows us to express $R$ in an explicit form.
    

    \begin{lem}     
        Let $v=(v_1,v_2,v_3) \in S^2$ with $v_1^2 + v_2^2 + v_3^2 = 1$. Then\\
        $$\bm{v} \times \bm{\nu} = \frac{(v_2,-v_1,0)}{\sqrt{1-v_3^2}}, \ \ \ \ (\bm{v} \times \bm{\nu}) \times \bm{v} = \frac{(-v_1v_3,-v_2v_3,1-v_3^2)}{\sqrt{1-v_3^2}}$$
        $$s = \frac{v_2x-v_1y}{\sqrt{1-v_3^2}}, \ \ \ t = v_1x + v_2y + v_3f(x,y), \ \ \ r = \frac{-v_1v_2x - v_2v_3y + (1-v_3^2)f(x,y)}{\sqrt{1-v_3^2}}$$
    \end{lem}

    \begin{proof}
        Since $g = (g \cdot (\bm{\nu} \times \bm{v})) \bm{\nu} \times \bm{v} + (g \cdot \bm{v}) \bm{v} + (g \cdot ((\bm{v} \times \bm{\nu}) \times \bm{v})) (\bm{v} \times \bm{\nu}) \times \bm{v}$, we conclude that

        $$
        \begin{array}{ccccccl}
             s &= &\vecTri{x}{y}{f(x,y)} &\cdot &\frac{1}{\sqrt{1-v^2_3}} \vecTri{v_2}{-v_1}{0} & = &\frac{v_2x-v_1y}{\sqrt{1-v_3^2}},\\
             t &= &\vecTri{x}{y}{f(x,y)} &\cdot &\vecTri{v_1}{v_2}{v_3} & = &v_1x + v_2y + v_3f(x,y),\\
             r &= &\vecTri{x}{y}{f(x,y)} &\cdot &\frac{1}{\sqrt{1-v^2_3}} \vecTri{-v_1v_3}{-v_2v_3}{1-v_3^2} & = &\frac{-v_1v_2x - v_2v_3y + (1-v_3^2)f(x,y)}{\sqrt{1-v_3^2}}.\\
        \end{array}
        $$
    \end{proof}
    

\section{Versality of the Rotation Unfolding}\label{S:Ver}

    In this section, we consider algebraic conditions when the rotation unfolding is versal. It is when $\theta (F)$ is spanned by $tF(\theta_2)$, $\omega F(\theta_3)$, and $V_R$, where $F = (x,y^2,f(x,y))$, $R$ is the rotation unfolding of $F$ and $V_R:= \langle \frac{\partial R}{\partial v_1}|_{\mathbb{R}^2 \times \{\bm{e_2}\}}, \frac{\partial R}{\partial v_3}|_{\mathbb{R}^n \times \{\bm{e_2}\}} \rangle _{\mathbb{R}}$.

    \begin{lem}     
        $$ V_{R} = \left\langle 
                \left( \begin{array}{c}
                    -y \\
                    2xy \\
                    0
                \end{array} \right),
                \left( \begin{array}{c}
                    0 \\
                    2yf(x,y) \\
                    -y
                \end{array} \right)
        \right\rangle_{\mathbb{R}}.$$
    \end{lem}

    \begin{proof}
        The generators of the subspace $V_R$ can be obtained by differentiating $R$ by each parameter and then substituting $\bm{v}$ by $\bm{e}_2=(0,1,0)$.
        
        $$
        \begin{array}{ccl}
             \left. \pardiff{R}{v_1} \right|_{\bm{v}=\bm{e}_2}  &=  & \left.\vecTri{\frac{-\frac{v_1}{v_2}x-y}{\sqrt{1-v^2_3}}}{2(v_1x+v_2y+v_3f)(x-\frac{v_1}{v_2}y)}{\frac{-v_3x + \frac{v_1}{v_2}v_3y}{\sqrt{1-v^2_3}}}\right|_{\bm{v}=\bm{e}_2}  \\
             &= & \left. \vecTri{-\frac{1}{\sqrt{1-v^2_3}} \left\{ \frac{v_1}{v_2}x + y \right\}}{2v_2 (\frac{v_1}{v_2}x + y + \frac{v_3}{v_2}f)(x-\frac{v_1}{v_2}y)}{-\frac{v_2}{\sqrt{1-v^2_3}}\frac{v_2}{v_2}(x - \frac{v_1}{v_2}y)}\right|_{\bm{v}=\bm{e}_2}\\
             &= & \vecTri{-y}{2xy}{0},\\
        \end{array}
        $$
        $$
        \begin{array}{ccl}
             \left.\pardiff{R}{v_3}\right|_{\bm{v}=\bm{e}_2}  &=  & \left. \vecTri{\frac{1}{1-v^2_3}(-\frac{v_3}{v_2}\sqrt{1-v^2_3}x - (v_2x - v_1y)(-\frac{v_3}{\sqrt{1-v^2_3}}))}{2(v_1x+v_2y+v_3f)(-\frac{v_1}{v_2}y+f)}{-\frac{v_2}{\sqrt{1-v^2_3}^3}x-\frac{v_2^2-v_3^2+v_3^4}{v_2\sqrt{1-v^2_3}^3}y-f\frac{v_3}{\sqrt{1-v^2_3}}}\right|_{\bm{v}=\bm{e}_2} \\
             & =    & \left.\vecTri{-\frac{v_1^2}{\sqrt{1-v^2_3}^3}\frac{v_1}{v_2}\frac{v_3}{v_2}(\frac{v_1}{v_2}x + y)}{2v_2 (\frac{v_1}{v_2}x + y + \frac{v_3}{v_2}f)(-\frac{v_1}{v_2}y+f)}{-\frac{v_2}{\sqrt{1-v^2_3}^3}\left[\frac{v_1}{v_2}x + \{ 1- (1-v^2_3) (\frac{v_3}{v_2})^2\}y - v_3(1-v^2_3)\frac{v_3}{v_2}f \right]}\right|_{\bm{v}=\bm{e}_2}\\
             & = & \vecTri{0}{2yf}{-y}.
        \end{array}
        $$
    \end{proof}

    The next proposition is for an algebraic condition for the rotation unfolding to be versal.

    \begin{prop}\label{prop:ACMThm}     
        Let $f(x,y) = \frac{1}{2} (k_1 x^2 + k_2 y^2) + \sum^m_{i+j \geq 3} \frac{1}{i!j!}a_{ij} x^i y^j + O(m+1)$ and $R$ be the rotation unfolding of the folding map $F$.\\
        (1) If $F$ is $S_0$ or $S_1$ singularity at $\bm{0}$, $R$ is always versal.\\
        (2) If $F$ is $S_2$ singularity at $\bm{0}$, $R$ is versal if and only if $k_1 - k_2 \not =0$.\\ 
        (3) If $F$ is $B_2$ singularity at $\bm{0}$, $R$ is versal if and only if $a_{13}(k_1-k_2) - 3a_{21} a_{12} \not = 0$
    \end{prop}

    \begin{proof}
        We use classification result by Proposition.\ref{prop:MC}. We suppose that the folding map $F=(x,y^2,f(x,y))$ is $\mathcal{A}$-equivalent to one of $S_1$, $S_2$ and $B_2$, since these $\mathcal{A}$-codimension are 2 or less. 
        Since the folding map $F$ is $k$-determined for $\mathcal{A}$, $m_n^{k+1} \theta(F) \subset T \mathcal{A}_e F$ (see Theorem. \ref{prop:IV}.), where $k=3,4$ and $5$, when $F$ is $\mathcal{A}$-equivalent to $S_1,S_2$ and $B_2$, respectively. It is thus enough to perform computations as $k$-jets.
        We investigate the folding map under the restriction given by Proposition.\ref{prop:AC}. 
        
        
        $$
        \begin{array}{l}
           \theta(F) = \ tF(\theta_2) + \omega F(\theta_3) + V_R\\
             \ = \left\langle \vecTri{1}{0}{f_x}, \vecTri{0}{2y}{f_y} \right\rangle_{\mathcal{E}_2} + \left\{ \vecTri{\eta_1(x,y^2,f)}{\eta_2(x,y^2,f)}{\eta_3(x,y^2,f)} \Bigg| \ \eta_i \in \mathcal{E}_3 \right\} + \left\langle \vecTri{-y}{2xy}{0}, \vecTri{0}{2yf}{-y} \right\rangle_{\mathbb{R}}\\
             \ = \left.\left\{ \phi_1\vecTri{1}{0}{f_x} + \phi_2\vecTri{0}{2y}{f_y} + \vecTri{\eta_1(x,y^2,f)}{\eta_2(x,y^2,f)}{\eta_3(x,y^2,f)} + r_1\vecTri{-y}{2xy}{0} + r_2\vecTri{0}{2yf}{-y} \right| 
            \begin{array}{l}
                 \phi_i \in \mathcal{E}_2,\ \eta_j \in \mathcal{E}_3  \\
                 r_1,r_2 \in \mathbb{R}\\
            \end{array}\right\}
        \end{array}
        $$
       Here $\mathcal{E}_n$ is a local ring of germs of a smooth function $f:\mathbb{R}^n \rightarrow \mathbb{R}$. For $\phi_1,\phi_2 \in \mathcal{E}_2$, we consider their Taylor expansions
       $$\sum b_{ij} x^iy^j,\ \ \ \sum c_{ij} x^iy^j$$
       respectively. Let $m_i$ be monomials of order $\leq k$. We represent each term above (as $k$-jets) as follows.

        $$
        \begin{array}{l}
           \phi_1 \vecTri{1}{0}{f_x} = \displaystyle \sum_{i,j,s,t,} \beta_{Ij}b_{st}m_i\bm{e}_j,\ \ \  \phi_2 \vecTri{0}{2y}{f_y} = \displaystyle \sum_{i,j,s,t,} \gamma_{Ij}c_{st}m_i\bm{e}_j\\
           \vecTri{\eta_1(x,y^2,f)}{\eta_2(x,y^2,f)}{\eta_3(x,y^2,f)} = q_1\vecTri{x}{0}{0} + \cdots + q_N \vecTri{0}{0}{f^n} = \displaystyle \sum_{i,j,s} \rho_{Ij}q_{s}m_i\bm{e}_j\\
           r_1 \vecTri{-y}{2xy}{0} + r_2 \vecTri{0}{2yf}{-y} = \displaystyle \sum_{i,j,s} \nu_{Ij} r_sm_i\bm{e}_j\\
        \end{array}
        $$
        where $n$ and $N$ are suitable natural numbers. If the rotation unfolding $R$ is versal, then there exists a $b_{st}, c_{st}, q_s$ and $v_s \in \mathbb{R}$ such that $\xi = \sum \beta_{Ij}b_{st}m_i\bm{e}_j + \sum \gamma_{Ij}c_{st}m_i\bm{e}_j + \sum \rho_{Ij}q_{s}m_i\bm{e}_j + \sum \nu_{Ij} r_sm_i\bm{e}_j$ for any $\xi \in \theta(g)$. Therefore, $R$ is versal if and only if the following matrix is of full rank.

        $$
        \begin{array}{c||ccc|ccc|ccc|ccc}
            & b_{00} &\cdots &b_{0k} &c_{00} &\cdots &c_{0k} &q_1 &q_2 &q_3 &r_1 & &r_2 \\\hline \hline
            
            m_1 \bm{e}_1 & & & & & & & & & & & \\
            \vdots & &\mbox{\Large{$B_1$} } & & &\mbox{\Large{$\Gamma_1$} } & & &\mbox{\Large{$\Theta_1$} } & & &\mbox{\Large{$P_1$}} & \\
            m_n \bm{e}_1 & & & & & & & & & & & \\ \hline
            
            m_1 \bm{e}_2 & & & & & & & & & & & \\
            \vdots & &\mbox{\Large{$B_2$} } & & &\mbox{\Large{$\Gamma_2$} } & & &\mbox{\Large{$\Theta_2$} } & & &\mbox{\Large{$P_2$} } & \\
            m_n \bm{e}_2 & & & & & & & & & & & \\  \hline
            
            m_1 \bm{e}_3 & & & & & & & & & & & \\
            \vdots & &\mbox{\Large{$B_2$} } & & &\mbox{\Large{$\Gamma_2$} } & & &\mbox{\Large{$\Theta_2$} } & & &\mbox{\Large{$P_2$} } & \\
            m_n \bm{e}_3& & & & & & & & & & & \\
        \end{array}
        $$
        Here, $B_j$ , $\Gamma_j$ , $\Theta_j$ and $P_j$ are the matrix $(\beta_{I_j} )_I$, $(\gamma_{I_j} )_I$, $(\theta_{I_j} )_I$, $(\rho_{I_j} )_I$, respectively. 
        

        We are going to describe the condition that the matrix above is of full rank in terms of $k_i$ and $a_{ij}$.
        For all $i,j \in \mathbb{N}$ a monomial $x^iy^{2j}\bm{e}_s$ is always in $\omega F(\theta_3)$, and we can perform computation modulo $\left \langle x^iy^{2j} \bm{e}_s \right \rangle_{\mathcal{E}_2}$. Furthermore, we need to consider whether $x^iy^{2j}f^l \bm{e}_s \in \omega F(\theta _3)$ represents by an element represented by 0, a monomial or a polynomial which is not a monomial, in $\theta (F) / \langle x^iy^{2j} \bm{e}_s \rangle_{\mathbb{R}} + \langle \langle x,y \rangle_{\mathcal{E}_2} ^{k+1} \rangle_{\mathbb{R}}$. If it is 0, there is nothing to consider, but if it is represented by a monomial, we perform computation modulo $\mathbb{R}$-module generated by it. 
        
        \underline{(1) About $S_1$ singularity} 
        Observe that $xf \in \omega F(\theta_3)$ is 0 and $f = \frac{1}{6}(3a_{21} x^2y + a_{03}y^3) \in \omega F(\theta_3)$ is a polynomial which is not a monomial. The matrix is shown below removing the blank rows and columns, and removing columns corresponding to the monomials above spanned by some elements $x^iy^{2j}f^l \bm{e}_s \in \omega F(\theta_3)$. We mark several non-zero components as candidates of pivot.

        $$
        \begin{array}{c||c@{\hskip2pt}c@{\hskip2pt}c@{\hskip2pt}c@{\hskip2pt}c@{\hskip2pt}c|c@{\hskip2pt}c@{\hskip2pt}c@{\hskip2pt}c@{\hskip2pt}c@{\hskip2pt}c|c@{\hskip2pt}c@{\hskip2pt}c|c@{\hskip2pt}c}
                &b_{00} &b_{10} &b_{01} &b_{11} &b_{21} &b_{03} &c_{00} &c_{10} &c_{01} &c_{20} &c_{11} &c_{02} &q_1 &q_2 &q_3 &r_1 &r_2 \\ \hline \hline
             y\bm{e}_1      & & &\fbox{1} & & & & & & & & & & & & &-1 & \\ 
             xy\bm{e}_1     & & & &\fbox{$1$} & & & & & & & & & & & & & \\ 
             x^2y\bm{e}_1   & & & & &\fbox{$1$} & & & & & & & &\frac{a_{21}}{2} & & & & \\ 
             y^3\bm{e}_1    & & & & & &\fbox{$1$} & & & & & & &\frac{a_{03}}{6} & & & & \\ \hline
             y\bm{e}_2      & & & & & & &\fbox{$2$} & & & & & & & & & & \\ 
             xy\bm{e}_2     & & & & & & & &\fbox{$2$} & & & & & & & &2 & \\ 
             x^2y\bm{e}_2   & & & & & & & & & & &\fbox{$2$} & & &\frac{a_{21}}{2} & & &k_1 \\ 
             y^3\bm{e}_2    & & & & & & & & & & & &2 & &\fbox{$\frac{a_{03}}{6}$} & & &k_2 \\ \hline
             y\bm{e}_3      & & & & & & &k_2 & & & & & & & & & &\fbox{$-1$} \\ 
             xy\bm{e}_3     &\fbox{$a_{21}$} & &k_1 & & & &a_{12} &k_2 & & & & & & & & & \\ 
             x^2y\bm{e}_3   &\frac{a_{31}}{2} &\fbox{$a_{21}$} &\frac{a_{30}}{2} &k_1 & & &\frac{a_{22}}{2} &a_{12} &\frac{a_{21}}{2} &k_2 & & & & &\frac{a_{21}}{2} & & \\ 
             y^3\bm{e}_3    &\frac{a_{13}}{6} & &\frac{a_{12}}{2} & & & &\frac{a_{04}}{6} & &\fbox{$\frac{a_{03}}{6}$} & & &k_2 & & &
             \frac{a_{03}}{2} & & \\ 
        \end{array}
        $$

        This matrix is of full rank since there are marked elements in all rows. Thus $R$ is versal if the folding map $F$ is equivalent to $S_1$.
        \rightline{$\square$}\\
        
        \underline{(2) About $S_2$ singularity} Observe that $x^2f,y^2f,f^2 \in \omega F(\theta_3)$ is 0, $xf = \frac{a_{03}}{6}xy^3 \in \omega F(\theta_3)$ is a monomial, and $f = \frac{a_{03}}{6}y^3 + \frac{a_{31}}{6} x^3y$ is a polynomial which is not a monomial. The matrix is shown below removing the blank rows and columns, and removing columns corresponding to the monomials above spanned by some elements $x^iy^{2j}f^l \bm{e}_s \in \omega F(\theta_3)$. We mark several non-zero components as candidates of pivot.

        $$
        \begin{array}{c||c@{\hskip2pt}c@{\hskip2pt}c@{\hskip2pt}c@{\hskip2pt}c@{\hskip2pt}c@{\hskip2pt}c|c@{\hskip2pt}c@{\hskip2pt}c@{\hskip2pt}c@{\hskip2pt}c@{\hskip2pt}c@{\hskip2pt}c|c@{\hskip2pt}c@{\hskip2pt}c|c@{\hskip2pt}c}
             & b_{00} & b_{10} & b_{01} & b_{11} & b_{21} & b_{03} & b_{31} & c_{00} & c_{10} & c_{01} & c_{20} & c_{11} & c_{02} & c_{30} & q_1 & q_2 & q_3 & r_1 & r_2 \\ \hline \hline
             y\bm{e}_1      & & &1 & & & &   & & & & & & &   & & &    &-1 & \\ 
             xy\bm{e}_1     & & & &\fbox{$1$} & & &   & & & & & & &   & & &    & & \\ 
             x^2y\bm{e}_1   & & & & &\fbox{$1$} & &   & & & & & & &   & & &    & & \\ 
             y^3\bm{e}_1    & & & & & &\fbox{$1$} &   & & & & & & &   &\frac{a_{03}}{6} & &    & & \\ 
             x^3y\bm{e}_1   & & & & & & &\fbox{$1$}   & & & & & & &   &\frac{a_{31}}{6} & &    & & \\ \hline
             
             y\bm{e}_2      & & & & & & &   &\fbox{$2$} & & & & & &   & & &    & & \\ 
             xy\bm{e}_2     & & & & & & &   & &2 & & & & &   & & &    &2 & \\ 
             x^2y\bm{e}_2   & & & & & & &   & & & & &\fbox{$2$} & &   & & &    & &\frac{k_1}{2}  \\ 
             y^3\bm{e}_2    & & & & & & &   & & & & & &\fbox{$2$} &   & &\frac{a_{03}}{6} &    & &\frac{k_2}{2}  \\ 
             x^3y\bm{e}_2   & & & & & & &   & & & & & & &\fbox{$2$}   & &\frac{a_{31}}{6} &    & &\frac{a_{30}}{6} \\ \hline
             
             y\bm{e}_3      & & & & & & &   &k_1 & & & & & &   & & &    & &\fbox{$-1$}  \\ 
             xy\bm{e}_3     & & &k_1 & & & &   &a_{12} &k_2 & & & & &   & & &    & &  \\ 
             x^2y\bm{e}_3   &\fbox{$\frac{a_{31}}{2}$} & &\frac{a_{30}}{2} &k_1 & & &   &\frac{a_{22}}{2} &a_{21} & &k_2 & & &   & & &    & &  \\ 
             y^3\bm{e}_3    &\frac{a_{13}}{6} & &\frac{a_{12}}{2} & & & &   &\frac{a_{04}}{6} & &\frac{a_{03}}{2}& & &k_2 &   & & &\fbox{$\frac{a_{03}}{6}$}    & &  \\ 
             x^3y\bm{e}_3   &\frac{a_{41}}{6} &\fbox{$\frac{a_{31}}{2}$} &\frac{a_{40}}{6} &\frac{a_{30}}{2} &k_1 & &   &\frac{a_{32}}{6} &\frac{a_{22}}{2} &\frac{a_{31}}{6} &a_{12} & & &k_2   & & &\frac{a_{31}}{6}    & & \\ 
        \end{array}
        $$

        The matrix is of full rank if and only if the following submatrix is of full rank because these elements are no mark.

        $$
        \begin{array}{c||c|c|c}
              &b_{00} &c_{10} &r_1 \\ \hline \hline
             y\bm{e}_1  &1    &   &-1 \\\hline
             xy\bm{e}_2 &    &2   &2 \\\hline
             xy\bm{e}_3 &k_1    &k_2   &
        \end{array}
        $$

        Therefore, the rotation unfolding $R$ is versal if only if this submatrix is of full rank, that is $(-k_1+k_2) \not = 0$.\\
        \rightline{$\square$}\\

        \underline{(3) About $B_2$ singularity} Observe that $x^3f, xy^2f, f^2 \in \omega F(\theta_3)$ is 0, $x^2f=\frac{a_{21}}{2}x^4y, \  y^2f=\frac{a_{21}}{2}x^2y^3, \ xf=\frac{a_{21}}{2}x^3y \in \omega F(\theta_3)$ is a monomial, and $f = \frac{a_{21}}{2} x^2y + \frac{a_{13}}{6} xy^3 + \frac{a_{05}}{120}y^5$ is a polynomial which is not a monomial.  The matrix is shown below removing the blank rows and columns, and removing columns corresponding to the monomials above spanned by some elements $x^iy^{2j}f^l \bm{e}_s \in \omega F(\theta_3)$. We mark several non-zero components as candidates of pivot.
        
        $$
        \begin{array}{c||c@{\hskip2pt}c@{\hskip2pt}c@{\hskip2pt}c@{\hskip2pt}c@{\hskip2pt}c@{\hskip2pt}c@{\hskip2pt}c@{\hskip2pt}c|c@{\hskip2pt}c@{\hskip2pt}c@{\hskip2pt}c@{\hskip2pt}c@{\hskip2pt}c@{\hskip2pt}c|c@{\hskip2pt}c@{\hskip2pt}c|c@{\hskip2pt}c}
             & b_{00} & b_{10} & b_{01} & b_{11} & b_{02} & b_{21} & b_{03} & b_{13} & b_{05} & c_{00} & c_{10} & c_{01} & c_{20} & c_{02} & c_{12} & c_{04} & q_1 & q_2 & q_3 & r_1 & r_2 \\ \hline \hline
             y\bm{e}_1      & & &1 & & & & & &   & & & & & & &   & & &    &-1 & \\ 
             xy\bm{e}_1     & & & &\fbox{$1$} & & & & &   & & & & & & &   & & &    & & \\ 
             x^2y\bm{e}_1   & & & & & &\fbox{$1$} & & &   & & & & & & &   &\frac{a_{21}}{2} & &    & & \\ 
             y^3\bm{e}_1    & & & & & & &\fbox{$1$} & &   & & & & & & &   & & &    & & \\ 
             xy^3\bm{e}_1   & & & & & & & &\fbox{$1$} &   & & & & & & &   &\frac{a_{13}}{6} & &    & & \\
             y^5\bm{e}_1    & & & & & & & & &\fbox{$1$}   & & & & & & &   &\frac{a_{05}}{120} & &    & & \\\hline
             
             y\bm{e}_2      & & & & & & & & &   &\fbox{$2$} & & & & & &   & & &    & & \\ 
             xy\bm{e}_2     & & & & & & & & &   & &2 & & & & &   & & &    &2 & \\ 
             x^2y\bm{e}_2   & & & & & & & & &   & & & &\fbox{$2$} & & &   & &\frac{a_{21}}{2} &    & &\frac{k_1}{2}  \\ 
             y^3\bm{e}_2    & & & & & & & & &   & & & & &\fbox{$2$} & &   & & &    & &\frac{k_2}{2}  \\ 
             xy^3\bm{e}_2   & & & & & & & & &   & & & & & &\fbox{$2$} &   & &\frac{a_{13}}{6} &    & &\frac{a_{21}}{2} \\
             y^5\bm{e}_2    & & & & & & & & &   & & & & & & &\fbox{$2$}   & &\frac{a_{05}}{120} &    & &\frac{a_{04}}{6} \\\hline
             
             y\bm{e}_3      & & & & & & & & &   &k_2 & & & & & &   & & &    & &\fbox{$-1$}  \\ 
             xy\bm{e}_3     &a_{21} & &k_1 & & & & & &   &a_{12} &k_2 & & & & &   & & &    & &  \\ 
             x^2y\bm{e}_3   &\frac{a_{31}}{2} &\fbox{$a_{21}$} &\frac{a_{30}}{2} &k_1 & & & & &   &\frac{a_{22}}{2} &a_{12} &\frac{a_{21}}{2} &k_2 & & &   & & &\frac{a_{21}}{2}    & &  \\ 
             y^3\bm{e}_3    &\frac{a_{13}}{6} & &\frac{a_{12}}{2} & & & & & &   &\frac{a_{04}}{6} & & & &k_2 & &   & & &    & &  \\ 
             xy^3\bm{e}_3   &\frac{a_{23}}{6} &\frac{a_{13}}{6} &\frac{a_{22}}{2} &\frac{a_{12}}{2} &\dbox{$a_{21}$} & &k_1 & &   &\frac{a_{14}}{6} &\frac{a_{04}}{6} &\frac{a_{13}}{2} & &a_{12} &k_2 &   & & &\dbox{$\frac{a_{13}}{6}$}    & & \\
             y^5\bm{e}_3    &\frac{a_{15}}{120} & &\frac{a_{14}}{24} & &\dbox{$\frac{a_{13}}{6}$} & &\frac{a_{12}}{2} & &   &\frac{a_{06}}{12} & &\frac{a_{05}}{24} & &\frac{a_{04}}{6} & &k_2   & & &\dbox{$\frac{a_{05}}{120}$}    & & \\
        \end{array}
        $$

        The dashed elements become trivial alternately with respect to rows because of $3a_{21}a_{05}-5a_{13}^2\not=0$. The matrix is of full rank if and only if the following submatrix is of full rank.

        $$
        \begin{array}{c||cc|c|c}
             &b_{00} &b_{01} &c_{10} &r_1 \\\hline \hline
             y\bm{e}_1      & &1 & &-1  \\\hline
             xy\bm{e}_2     & & &2 &2   \\\hline
             xy\bm{e}_3     &a_{21} &k_1 &k_2 & \\
             y^3\bm{e}_3    &\frac{a_{13}}{6} &\frac{a_{12}}{2} & & \\
        \end{array}
        $$
        
        Therefore, the rotation unfolding $R$ is versal if only if this matrix is of full-rank, that is $3a_{21} a_{12} - a_{13} (k_{1}-k_{2})\not = 0$.
        \end{proof}


    \section{Geometric Interpretation}\label{S:Geo}

    \subsection{Geometric condition of versality}
        We interpret the condition of the coefficients for rotation unfolding to be versal in terms of the geometrical properties of the surface $(x,y,f(x,y))$.
        
        \begin{thm}     
            Let $f(x,y) = \frac{1}{2} (k_1 x^2 + k_2 y^2) + \sum^m_{i+j \geq 3} \frac{1}{i!j!}a_{ij} x^i y^j+O(m+1)$, and $R$ be the rotation unfolding of the folding map $F$.\\
            (1) If $F$ is $S_0$ or $S_1$ singularity at $\bm{0}$, $R$ is always versal.\\
            (2) If $F$ is $S_2$ singularity at $\bm{0}$, $R$ is versal if and only if the origin is not umbilic.\\ 
            (3) If $F$ is $B_2$ singularity at $\bm{0}$, $R$ is versal if and only if the origin is not umbilic and v-ridge line is transverse to the plane $\Pi$, or $D_4$ type umbilic.
        \end{thm}
        
        \begin{proof}
            The origin is umbilic if and only if $k_1 - k_2 = 0$, so (2) is trivial by Proposition \ref{prop:ACMThm}. We discuss the statement (3). If the origin is not umbilic, the $v_2$-ridge line is given by the following: (see Lemma \ref{prop:ridge}.)
            $$\bm{v}_2 \kappa_2(u,v) = a_{03} + \frac{1}{k_2-k_1} \{ 3a_{21}a_{12} - a_{13}(k_1-k_2)\}u + \frac{1}{k_2-k_1} \{ 3a_{12}^2 + (a_{04}-3k_2^3)(k_2-k_1) \} v +O(2) = 0.$$
            
            Thus, the $R$ is versal if and only if the coefficient of $v$ is not $0$. That is a ridge line transverse to the plane $\Pi$. If the origin is umbilic, i.e., $k_1-k_2 = 0$, the versality condition of $R$ is able to be rewritten as $3a_{21} a_{12} \not = 0$. In addition, the coefficient $a_{21} \not = 0$ because $R$ is $B_2$ singularity (Proposition \ref{prop:AC}), so we get the condition that $a_{12} \not = 0.$ Then, we represented the function $f(x,y)$ by complex coordinate, and the complex polynomial $c(z)$ denote the its cubic form, i.e., $c(z) = 1/6 (\alpha z^3 + 3\beta z^2\overline{z} + 3\bar{\beta} z \bar{z}^2 + \bar{\alpha}\bar{z}^3)$. We chose $w \in \mathbb{C}$ such that $|w| = 1$, and that
            $$c(wz) = \frac{1}{6}(a_{30}x^3 + 3a_{21} x^2y + 3a_{12} xy^2 + a_{03} y^3).$$

            Then, we get the coefficients of $c(wz)$ as follows:

            $$a_{30} = c(wz) \big|_{z=1}, \ \ \ a_{21} = \pardiff{c(wz)}{y}\big|_{z=1}, \ \ \ a_{12} = \pardiff{c(wz)}{x}\big|_{z=i},\ \ \ a_{03} = c(wz)\big|_{z=i}.$$

            Now we assume $R$ is not versal, i.e., $a_{12} = 0$, and $(x,y^2,f(x,y))$ is $B_2$ singularity, i.e., $a_{03} = 0$. Besides, if the umbilic is not of type $D_4$, $c(wz)=\pardiff{c(wz)}{x}=0$. (see Proposition.\ref{prop:umb})
        \end{proof}
            
            The resultant of $a_{12}$ and $a_{03}$ should be $0$ if there is above $w$.

            $$
            \begin{array}{ccc}
                 a_{12} & = & \frac{i}{6}(\alpha w^3 - \beta w^2 \bar{w} - \bar{\beta} w \bar{w}^2 + \bar{\alpha} \bar{w}^3)\\
                 a_{03} & = & -\frac{1}{6}(\alpha w^3 - 3\beta w^2 \bar{w} + 3\bar{\beta} w \bar{w}^2 - \bar{\alpha} \bar{w}^3)
            \end{array}
            $$

            and the resultant of them is

            $$
            \begin{array}{ccl}
                 \text{Res}(a_{12},a_{03}) &= & (\frac{i}{6})^3\cdot(-\frac{1}{6})^3\times
                 \begin{vmatrix}
                \alpha & -\beta & -\bar{\beta} & \bar{\alpha} & & \\
                & \alpha & -\beta & -\bar{\beta} & \bar{\alpha} & \\
                & & \alpha & -\beta & -\bar{\beta} & \bar{\alpha} \\
                \alpha & -3\beta & 3\bar{\beta} & -\bar{\alpha} & & \\
                & \alpha & -3\beta & 3\bar{\beta} & -\bar{\alpha} & \\
                & & \alpha & -3\beta & 3\bar{\beta} & -\bar{\alpha} \\
                \end{vmatrix} \\ \\
                 & = & \frac{i}{6^6} \cdot |\alpha|^2 \times \left\{|\alpha|^4 -3 |\beta|^4 - 6|\alpha|^2|\beta|^2 + 4(\alpha \bar{\beta}^3 + \bar{\alpha} \beta^3) \right\} \\
            \end{array}
            $$

            On the plane $\alpha=1$, we draw a graph of $\text{Res}(a_{12},a_{03})=0$. (The dashed is a circle of radius 1.) This graph corresponds to an inner hypocycloid in the figure of Proposition \ref{prop:umb}.\\

            \begin{figure}[h!]
                \centering
                \includegraphics{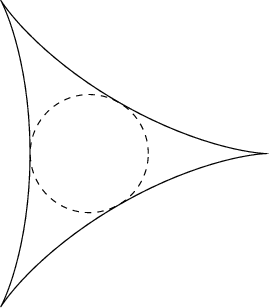}
                \caption{the locus of $\text{Res}(a_{12},a_{03})=0$}
            \end{figure}

    \subsection{Bifurcation Set of the Rotation Unfolding}\label{S:BofR}
        In this section, we discuss bifurcation sets of the rotation unfolding. The rotation unfolding is considered as a subfamily of the folding family defined by Bruce and Wilkinson \cite[$\S 2$]{BW}. The folding family is parameterized by planes in $\mathbb{R}^3$ and the bifurcation set of the folding family is the dual of the focal and symmetry sets, that is, the set of tangent planes of the focal set and symmetry set (see \cite[Proposition 2.3]{BW}). 

        Now we show the bifurcation set of versal unfoldings of $S_1,S_2$, and $B_2$ singularity \cite[Fig.2.]{BW}. The self-tangent part is dashed.

            \begin{figure}[h!]
                \centering
                    \includegraphics{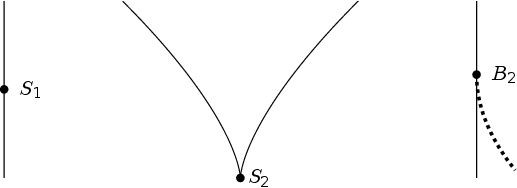}
                \caption{Bifurcation sets}
            \end{figure}

        The assertions above are consequences of the following proposition.

        \begin{prop}
            Let $a \in \mathbb{R}$. The versal unfolding of $S_2$ and $B_2$ are $(x,y^2,y^3 - x^3y + sy + txy)$ and $(x, y^2, y^5 - x^2y + sy + ty^3)$, respectively. In this case, the position of the singularity and the locus of the bifurcation set are as follows.\\
            \ronumi The locus of a mono-germ is $(s,t) = (-2a^3,3a^2)$ at $(a,0)$ in the source of the folding map, and there is no locus of a bi-germ.\\
            \ronumii The locus of a mono-germ is $(s,t) = (0,a)$ at $(0,0)$ in the source of the folding map, and the locus of a bi-germ is $(s,t) = (a^4,-2a^2)$ at $(0,a),(0,-a)$ in the source of the folding map.
        \end{prop}

        \begin{proof}
            It is easy for this proposition to prove, so we show an outline of the proof. First, we calculate $\small\fracset{\theta(f)}{T\mathcal{A}_ef}$, and versal unfolding is obtained for the folding map to add a linear combination of these generators over $\mathbb{R}$. We denote this versal unfolding by $\tilde{R}$.
            
            We consider the cases when there are points that are unstable points about mono-germ or bi-germ. If mono-germ, we will find the condition for the versal unfolding parameters when $\tilde{R}_x \times \tilde{R}_y=0$ has multiple roots since it has no multiple root generally. If bi-germ, these tangent planes have self-tangency. Thus, the following holds since these are equal to images of the differential map for versal unfolding $R$ at $(x,y)$ and $(x,-y)$.
            
        $$
        \begin{array}{lcc}
            \text{Im } dR_{(x,y)} = \text{Im } dR_{(x,-y)} \\
         
            \ \Leftrightarrow \ \left\langle \tilde{R}_x(x,y), \tilde{R}_y(x,y) \right\rangle_{\mathbb{R}} = \left\langle \tilde{R}_x(x,-y), \tilde{R}_y(x,-y) \right\rangle_{\mathbb{R}}\\
         
            \ \Leftrightarrow \ \left\{ \tilde{R}_x(x,y) \times \tilde{R}_y(x,y) \right\} \times \left\{ \tilde{R}_x(x,-y) \times \tilde{R}_y(x,-y) \right\} = 0 \\
        \end{array}
        $$
        Therefore, the condition is the solution of $(\tilde{R}_x(x,y) \times \tilde{R}_y(x,y))\times(\tilde{R}_x(x,-y) \times \tilde{R}_y(x,-y))=0$.\\
    \end{proof}

        Thus, a point of the bifurcation set of the rotation unfolding corresponds to a tangent plane of the focal/symmetry set of $M$, which contains the origin.
        
        \begin{itemize}
            \item If the folding map $F$ is $S_1$ singularity at the origin, the bifurcation set of the rotation unfolding is a non-singular curve and its point corresponds to a locus of tangent planes of the focal set of the surface $M$, which is passing through the origin.
            
            \item If the folding map $F$ is $S_2$ singularity at the origin and the origin is not umbilic of the surface $M$, the bifurcation set of the rotation unfolding is equivalent to a curve having (2,3)-cusp and its point corresponding to a locus of tangent planes of the focal set of the surface $M$, which is passing through the origin.
            
            \item If the folding map $F$ is $B_2$ singularity at the origin, and the origin is not umbilic of $M$ and $v$-ridge line is transverse to the plane $\Pi$, or $D_4$ type umbilic of M, the bifurcation set of the rotation unfolding is a union of a non-singular curve $C$ and a segment which tangents to $C$ at the origin, and its point corresponds to a locus of tangent planes of the focal set or symmetry set of the surface $M$, which is passing through the origin.
            
        \end{itemize}

\end{document}